\documentclass[12pt]{amsart}

\usepackage{amssymb}
\usepackage{bm}
\usepackage{enumerate}
\usepackage{graphicx}
\usepackage{psfrag}
\usepackage{color}
\usepackage{url}

\usepackage{algpseudocode}

\usepackage{mathtools}

\usepackage{xy}
\input xy
\xyoption{all}

\newcommand{\excise}[1]{}

\newtheorem{thm}{Theorem}[section]
\newtheorem{lemma}[thm]{Lemma}

\newtheorem{prop}[thm]{Proposition}

\newtheorem{question}[thm]{Question}

\theoremstyle{definition}

\newtheorem{example}[thm]{Example}
\newtheorem{remark}[thm]{Remark}
\newtheorem{defn}[thm]{Definition}

\numberwithin{equation}{section}

\newcommand{\ring}[1]{\ensuremath{\mathbb{#1}}}

\newcommand\ZZ{\ring{Z}}

\begin{document}

\mbox{}
\title[On the periodicity of irreducible elements in ACMs]{On the periodicity of irreducible elements in arithmetical congruence monoids}

\author{Jacob Hartzer}
\address{Mathematics Department\\Texas A\&M University\\College Station, TX 77843}
\email{jmhartzer@tamu.edu}

\author{Christopher O'Neill}
\address{Mathematics Department\\Texas A\&M University\\College Station, TX 77843}
\email{coneill@math.tamu.edu}

\date{\today}

\begin{abstract}
Arithmetical congruence monoids, which arise in non-unique factorization theory, are multiplicative monoids $M_{a,b}$ consisting of all positive integers $n$ satsfying $n \equiv a \bmod b$.  In this paper, we examine the asymptotic behavior of the set of irreducible elements of $M_{a,b}$, and characterize in terms of $a$ and $b$ when this set forms an eventually periodic sequence.  
\end{abstract}

\maketitle


\section{Introduction}
\label{s:intro}


Fix positive integers $a$ and $b$, and consider the set $M_{a,b}$ of positive integers $n$ satisfying the equation $n \equiv a \bmod b$.  If $M_{a,b}$ is closed under multiplication, it is known as an \emph{arithmetical congruence monoid} (Definition~\ref{d:acm}).  Since their introduction~\cite{acmarithmetic}, much of the literature concerning arithmetical congruence monoids has centered around their factorization structure~\cite{acmdelta,acmarithmetic,acmgenelast}, that is, the different ways in which monoid elements can be expressed as products of irreducible elements.  Unlike the set $\ZZ_{\ge 1}$ of all positive integers, which admits unique factorization into primes, factorization of elements of $M_{a,b}$ need not be unique; see Example~\ref{e:hilbert}.  

The class of arithmetical congruence monoids encompasses a wide range of factorization structures.  Some are Krull monoids, which have particularly well-behaved factorization structure~ \cite{nonuniq}, while others are ill-behaved enough to have non-accepted elasticity~\cite{acmacceptelast}, a pathological factorization property found in few ``naturally occuring'' monoids.  Additionally, the factorization structure of arithmetical congruence monoids is strongly connected to prime factorization in the integers, one of the classical motivations of broader factorization theory.  For a thorough overview of the literature on arithmetical congruence monoids, see Baginski and Chapman's survey article~\cite{acmsurvey}.  

In this paper, we examine the distribution of irreducible elements in arithmetical congruence monoids.  Our main result (Theorem~\ref{t:maincharacterization}) gives a complete answer to Question~\ref{q:periodicity}, which appeared as \cite[Open Question~4.6]{acmsurvey} in the aforementioned survey article.  

\begin{question}\label{q:periodicity}
Is the set of irreducibles elements of $M_{a,b}$ eventually periodic?  
\end{question}

\begin{thm}\label{t:maincharacterization}
The irreducible elements of $M_{a,b}$ form an eventually periodic sequence if and only if $a | b$.  
\end{thm}

The proof of Theorem~\ref{t:maincharacterization} is split between Sections~\ref{s:periodic} and~\ref{s:aperiodic}, each of which include one direction of the proof as Theorems~\ref{t:periodicacms} and~\ref{t:aperiodicacms}, respectively.  Initial investigations into Question~\ref{q:periodicity}, as well as the formation of several proofs in this paper, made use of a new software package for working with arithmetical congruence monoids; see Remark~\ref{r:sagepackage}.

\section{Arithmetical congruence monoids}
\label{s:background}

\begin{defn}\label{d:acm}
An \emph{arithmetical congruence monoid} is a multiplicative submonoid of $(\ZZ_{\ge 1}, \cdot)$ of the form 
$$M_{a,b} = \{1\} \cup \{n \in \ZZ_{\ge 1} : n \equiv a \bmod b\}$$
for positive integers $a, b \in \ZZ_{\ge 1}$ satisfying $a^2 \equiv a \bmod b$.  An element $u \in M_{a,b}$ is \emph{irreducible} if it cannot be written as a product of two non-unit elements of $M_{a,b}$.  
A \emph{factorization} of a given element $n \in M_{a,b}$ is an expression of the form
$$n = u_1 \cdots u_r$$
for irredudible elements $u_1, \ldots, u_r \in M_{a,b}$.  
\end{defn}

\begin{remark}\label{r:acm}
The condition $a^2 \equiv a \bmod b$ in Definition~\ref{d:acm} simply ensures that $M_{a,b}$ is closed under multiplication.  We include the identity element $1 \in M_{a,b}$, as doing so does not affect the factorization structure, but allows many statements to be simplified.  
\end{remark}

\begin{example}\label{e:hilbert}
Consider the arithmetical congruence monoid $M_{1,4}$, known as the \emph{Hilbert monoid}.  Any prime integer $p$ satisfying $p \equiv 1 \bmod 4$ is irreducible in $M_{1,4}$, but these are not the only irreducible elements of $M_{1,4}$.  For instance, $9, 21, 49 \in M_{1,4}$ are all irreducible, since each is the product of two prime integers that lie outside of $M_{1,4}$.  As a result, some elements of $M_{1,4}$ have multiple distinct factorizations (in the sense of Definition~\ref{d:acm}).  For example, $441 \in M_{1,4}$ has two distinct factorizations: 
$$441 = 9 \cdot 49 = 21^2.$$
Non-unique factorizations occur in every arithmetical congruence monoid, with the exception of $M_{1,1}$ and $M_{1,2}$.  
\end{example}

\begin{remark}\label{r:sagepackage}
Proving results involving arithmetical congruence monoids often requires locating monoid elements with specific factorization properties (see, for instance, the proofs of \cite[Theorems~4.8 and~4.9]{acmsurvey}, or Proposition~\ref{p:longreduciblesequences} in this paper).  To aid in this process, the authors developed a Sage \cite{sage} package for computing factorizations in arithmetical congruence monoids.  This package is now publicly available under the MIT license, and can be downloaded from the following webpage: 
\begin{center}
\url{http://www.math.tamu.edu/~coneill/acms}
\end{center}
This package was used for initial investigations into Question~\ref{q:periodicity}, and to aid in locating sequences of elements necessary to prove Proposition~\ref{p:longreduciblesequences}.  It was also used to generate the plots included in Figures~\ref{f:periodic} and~\ref{f:aperiodic}.  It is the author's hope that others interested in studying the factorization properties of arithmetical congruence monoids will find this package useful as well.  
\end{remark}

We conclude this section by recalling the following elementary fact from number theory, which will be used in the proof of Theorem~\ref{t:aperiodicacms}.  

\begin{thm}[Dirichlet]\label{t:dirichlet}
If $a$ and $b$ are relatively prime positive integers, then there are infinitely many primes $p$ satisfying $p \equiv a \bmod b$.  
\end{thm}

\section{The periodic case}
\label{s:periodic}

The main result of this section is Theorem~\ref{t:periodicacms}, which proves the direction of Theorem~\ref{t:maincharacterization} concerning arithmetical congruence monoids whose irreducible elements form an eventually periodic sequence.  Most of the argument is contained in Lemma~\ref{l:reducible}, which gives two modular conditions (one necessary and one sufficient) for reducibility in any arithmetical congruence monoid.  Lemma~\ref{l:reducible} will also be used in the proof of Theorem~\ref{t:aperiodicacms}, which proves the other direction of Theorem~\ref{t:maincharacterization}.  

Before stating these results, we give an example.

\begin{example}\label{e:periodic}
Consider the arithmetical congruence monoid $M_{7,42}$, whose irreducible and reducible elements are depicted in Figure~\ref{f:periodic}.  As is readily visible from the plot, exactly one in every 7 elements is reducible, namely the following elements: 
$$49, 343, 637, 931, 1225, \ldots$$
In particular, the reducible elements of $M_{7,42}$ form an arithmetic sequence with step size $7 \cdot 42 = 294$, which is the period guaranteed by Theorem~\ref{t:periodicacms}.  Moreover, every reducible element is congruent to $7^2 = 49$ modulo $294$, as guaranteed by Lemma~\ref{l:reducible}.  
\end{example}

\begin{figure}
\begin{center}
\includegraphics[width=6.0in]{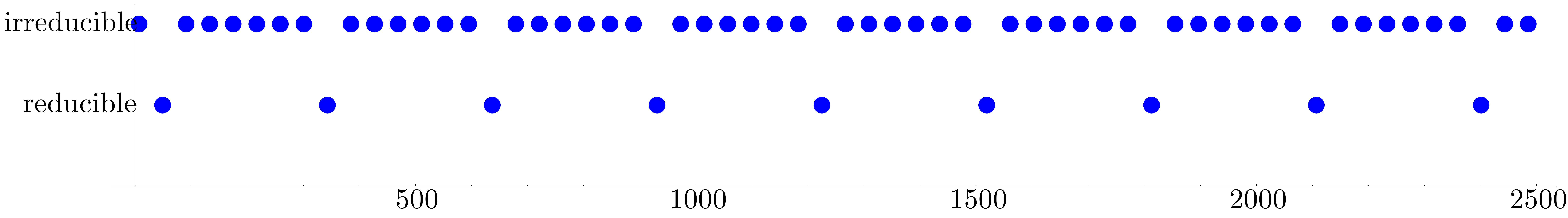}
\end{center}
\caption{A plot depicting the irreducible elements of $M_{7,42}$ from Example~\ref{e:periodic}.}
\label{f:periodic}
\end{figure}

\begin{lemma}\label{l:reducible} 
Fix an arithmetical congruence monoid $M_{a,b}$, and let $g = \gcd{(a,b)}$.  
\begin{enumerate}[(a)]
\item 
Any reducible element $n \in M_{a,b}$ satisfies $n \equiv a^2 \bmod bg$.  

\item 
If $a > 1$, then any monoid element $n \equiv a^2 \bmod ab$ is reducible.  
\end{enumerate}
\end{lemma}

\begin{proof}
For $q_1, q_2 \in \ZZ_{\ge 0}$, multiplying the elements $a + q_1b, a + q_2b \in M_{a,b}$ yields
$$\begin{array}{rcl}
(a + bq_1)(a + bq_2)
&=& a^2 + ab(q_1 + q_2) + b^2q_1q_2 \\
&=& a^2 + bg(\frac{a}{g}(q_1 + q_2) + \frac{b}{g}q_1q_2)
\end{array}$$
which is congruent to $a^2$ modulo $bg$.  This proves part~(a).  For part~(b), suppose $a > 1$ and fix $n = a^2 + abq \in M_{a,b}$.  Since $a \ne 1$, writing
$$n = a^2 + abq = a(a + bq).$$
expresses $n$ as a product of nonunits, thus proving $n$ is reducible.  
\end{proof}

\begin{thm}\label{t:periodicacms}
Fix an arithmetical congruence monoid $M_{a,b}$ satisfying $a > 1$ and $a | b$.  An element $n \in M_{a,b}$ is reducible if and only if $n \equiv a^2 \bmod ab$.  
\end{thm}

\begin{proof}
Both directions follow from Lemma~\ref{l:reducible} since $\gcd(a,b) = a$ in this case.  
\end{proof}

\section{The aperiodic case}
\label{s:aperiodic}

In this section, we complete the proof of Theorem~\ref{t:maincharacterization} by showing that the set of irreducible elements of any arithmetical congruence monoid not covered by Theorem~\ref{t:periodicacms} is not eventually periodic (Theorem~\ref{t:aperiodicacms}).  The main idea of the proof is given in Proposition~\ref{p:longreduciblesequences}, which produces arbitrarily long sequences of reducible elements in any such arithmetical congruence monoid.  

\begin{example}\label{e:aperiodic}
Depicted in Figure~\ref{f:aperiodic} are the irreducible elements of the arithmetical congruence monoid $M_{9,12}$.  The large red dots indicate elements of the set 
$$A = \{n \in M_{9,12} : n \equiv 9^2 \bmod 36\}$$
defined in Proposition~\ref{p:longreduciblesequences}.  Notice that every element of $M_{9,12}$ that lies outside of $A$ (represented by a small blue dot) is irreducible, as predicted by Lemma~\ref{l:reducible}(b).  

Proposition~\ref{p:longreduciblesequences} provides the first step in the proof of Theorem~\ref{t:aperiodicacms} by locating arbitrarily long sequences of reducible elements in the set $A$.  For $M_{9,12}$, the following sequence of $k = 4$ consecutive elements of $A$ is identified.  
$$\begin{array}{r@{}c@{}l@{\qquad}r@{}c@{}l}
31995873 &{}={}& 21 \cdot 1523613
&
31995909 &{}={}& 33 \cdot 969573
\\
31995945 &{}={}& 45 \cdot 711021
&
31995981 &{}={}& 57 \cdot 561333
\end{array}$$
The remainder of the proof of Theorem~\ref{t:aperiodicacms} demonstrates that, under certain conditions, the set $A$ also contains infinitely many irreducible elements.  
\end{example}

\begin{figure}
\begin{center}
\includegraphics[width=6.0in]{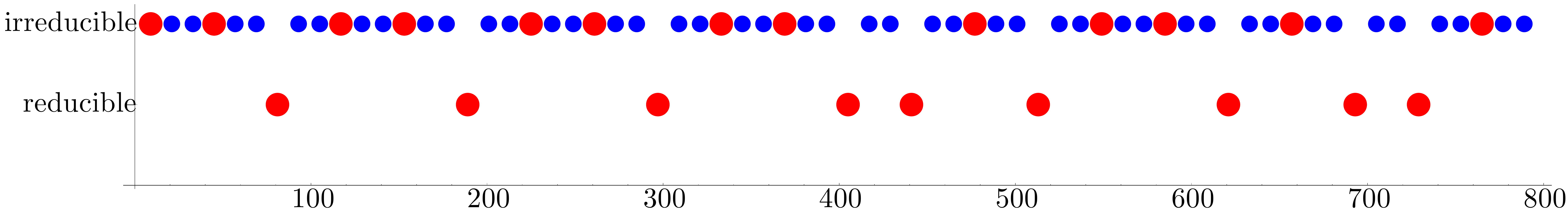}
\end{center}
\caption{A plot depicting the irreducible elements of $M_{9,12}$ from Example~\ref{e:aperiodic}.  
}
\label{f:aperiodic}
\end{figure}

\begin{prop}\label{p:longreduciblesequences}
Fix an arithmetical congruence monoid $M_{a,b}$, and let $g = \gcd(a,b)$.  The set 
$$A = \{n \in M_{a,b} : n \equiv a^2 \bmod bg\} \subset M_{a,b}$$
contains arbitrarily long sequences of consecutive reducible elements.  
\end{prop}

\begin{proof}
Fix $k \ge 1$, and consider the sequence $n_1, \ldots, n_k$ given by
$$n_j = bgj + ga + (a + b - g)\prod_{i=1}^k (bi + a)$$
for $j = 1, \ldots, k$.  We first verify that $n_j \in A$ for each $j \le k$.  Since $a^2 \equiv a \bmod b$, we have $a^k = a + bq$ for some $q > 0$.  This implies
$$\begin{array}{rcl}
n_j
&\equiv& ga + (a + b - g)a^k \bmod bg \\
&\equiv& a^{k+1} + g(a - a^k) \bmod bg \\
&\equiv& a(a + bq) \bmod bg \\
&\equiv& a^2 \bmod bg,
\end{array}$$
meaning $n_j \in A$.  Now, to prove that each $n_j$ is reducible in $M_{a,b}$, we express $n_j$ as
$$n_j = (bj + a)\left(g + (a + b - g)\textstyle\prod_{i \ne j} (bi + a)\right).$$
Clearly the first factor lies in $M_{a,b}$.  Moreover, $a(a - 1)$ divides $b$ and $\gcd(\frac{a}{g},b) = 1$, so $g(a-1)$ must also divide $b$.  This means the second factor satisfies 
$$\begin{array}{rcl}
g + (a + b - g)\prod_{i \ne j} (bi + a)
&\equiv& g + (a - g)a \bmod b \\
&\equiv& a^2 + g(1 - a) \bmod b \\
&\equiv& a \bmod b
\end{array}$$
and thus also lies in $M_{a,b}$.  This completes the proof.  
\end{proof}

\begin{remark}\label{r:longreduciblesequences}
Resume notation from the proof of Proposition~\ref{p:longreduciblesequences} above.  The given sequence of reducible elements can be easily generalized.  In particular, for each $c \ge 0$, the elements given by
$$bgj + g(bc + a) + (a + b - g)\prod_{i=1}^k (bi + bc + a)$$
for $j = 1, \ldots, k$ also form a sequence of consecutive reducible elements in $A$.  The proof of this fact is analogous to the proof of Proposition~\ref{p:longreduciblesequences} given above, where the given sequence coincides with the special case $c = 0$.  
\end{remark}

\begin{remark}\label{r:longreduciblesequenceassumptions}
Proposition~\ref{p:longreduciblesequences} makes no assumptions on the arithmetical congruence monoid $M_{a,b}$.  In particular, if $a > 1$ and $a | b$ as in Theorem~\ref{t:periodicacms}, then every element of $A$ is reducible by Lemma~\ref{l:reducible}.  In all other cases, however, one can find infinitely many irreducible elements in the set $A$.  This is the content of Theorem~\ref{t:aperiodicacms} below.  
\end{remark}

\begin{thm}\label{t:aperiodicacms}
Fix an arithmetical congruence monoid $M_{a,b}$, and let $g = \gcd(a,b)$.  If either $a = 1$ or $g < a$, then the irredudible elements of $M_{a,b}$ do not form an eventually periodic sequence.  
\end{thm}

\begin{proof}
Let $A$ denote the set defined in Proposition~\ref{p:longreduciblesequences}.  By Lemma~\ref{l:reducible}(a), every element of $M_{a,b} \setminus A$ is irreducible, and by Proposition~\ref{p:longreduciblesequences}, the set $A$ contains arbitrarily long sequences of reducible elements of $M_{a,b}$.  To complete the proof, it suffices to show that $A$ also contains infinitely many irreducible elements of $M_{a,b}$.  

Fix a prime integer of the form $p = (\frac{a}{g})^2 + \frac{b}{g}q$ for $q \in \ZZ_{\ge 0}$.  If $a = 1$, then $p = 1 + bq$ lies in $A$ and is clearly irredudible in $M_{a,b}$.  Otherwise, consider the monoid element 
$$n = g^2p = a^2 + bgq \in M_{a,b}.$$
The above expression implies $n \in A$, and since every element of $M_{a,b}$ is divisible by $g$, the only possible factorization of $n$ in $M_{a,b}$ is $n = (g)(gp)$, the first factor of which lies outside of $M_{a,b}$ by assumption.  As such, $n$ is irreducible in $M_{a,b}$.  

We now complete the proof by applying Theorem~\ref{t:dirichlet}, which ensures there are infinitely many primes $p$ satisfying $p \equiv (\frac{a}{g})^2 \bmod \frac{b}{g}$ since $\gcd((\frac{a}{g})^2, \frac{b}{g}) = 1$.  
\end{proof}

Together, Theorems~\ref{t:periodicacms} and~\ref{t:aperiodicacms} provide a complete proof of Theorem~\ref{t:maincharacterization}.  

\begin{proof}[Proof of Theorem~\ref{t:maincharacterization}]
Apply Theorems~\ref{t:periodicacms} and~\ref{t:aperiodicacms}.  
\end{proof}




\end{document}